\documentclass[11pt,letterpaper]{amsart}

\usepackage{amsmath}
\usepackage{amssymb} 
\usepackage[all,cmtip]{xy}
\usepackage{fancyvrb}
\usepackage{hyperref}

\usepackage{stmaryrd} 

\newcommand{\Z}{{\mathbb Z}} 
\newcommand{\Q}{{\mathbb Q}} 
\newcommand{\C}{{\mathbb C}}

\newcommand{\cO}{{\mathcal O}} 
\newcommand{\cF}{{\mathcal F}}

\newcommand{\cG}{{\mathcal G}} 
 
\newcommand{\cH}{{\mathcal H}}

\DeclareMathOperator{\ord}{ord}

\newcommand{\gM}{{\mathfrak M}}

\newcommand{\gr}{{\mathfrak r}}

\usepackage{amsthm}
\usepackage[theoremfont]{newtxtext} 
\usepackage{newtxmath} 

\newtheorem{thm}{Theorem}

\newtheorem*{Mazur_5} {Theorem (Mazur)}
\newtheorem{prop}{Proposition}

\theoremstyle{remark}

\newtheorem*{rem}{Remark}

\DeclareFontFamily{U}{wncy}{}
    \DeclareFontShape{U}{wncy}{m}{n}{<->wncyr10}{}
    \DeclareSymbolFont{mcy}{U}{wncy}{m}{n}
    \DeclareMathSymbol{\Sh}{\mathord}{mcy}{"58}  

\newcommand{\itop}[2]{\genfrac {}{}{0pt}{3}{#1}{#2} }

\begin{document}

\author[P. Guerzhoy]{P. Guerzhoy}
\address{ 
Department of Mathematics,
University of Hawaii, 
2565 McCarthy Mall, 
Honolulu, HI,  96822-2273 
}
\email{pavel@math.hawaii.edu}

\title[]{Congruences for traces of singular moduli and Hurwitz - Kronecker class numbers}
\begin{abstract}

Traces of singular moduli were introduced and studied by Zagier in 1998. Being simultaneously the (traces of) values of a modular function ($j$-invariant) and Fourier coefficients of modular forms 
- which constitutes Zagier's duality - these integers are quite interesting. 
Since then, a substantial amount of research was devoted to various properties of these numbers, congruences in particular. 
We present an alternative point of view on these congruences, specifically, we view them as congruences between certain weight $3/2$ modular forms under repeated action of $U$-operator. 
That allows us to obtain a general result which includes some previously known results as special cases. Our approach is especially effective when the prime modulus is relatively small.
In these cases, we obtain explanations for certain numerical observations and quantification of some previously known qualitative results. 
As an application, we obtain  modulo $11$ congruences between the traces of singular moduli and class numbers of quadratic fields in the case when the twisted central special value of the $L$-function associated with the elliptic curve of conductor $11$ vanishes. 

\end{abstract}

\maketitle

\section{Introduction} \label{sec_intro}

Throughout, $D \geq 1$ and $d \geq 0$ are  integers satisfying $D \equiv 0,1 \pmod 4$ and $d \equiv 0,3 \pmod 4$.

Traces (twisted whenever $D>1$) of singular moduli $B(D,d) \in \Z$ were introduced by Zagier in \cite{Zagier_traces}.  

Denote by $H(d)$ the Hurwitz - Kronecker class number (which coincides, in the case when  $-d\neq -3,-4$ is fundamental, with the class number of  $\Q\left(\sqrt{-d}\right)$, see \eqref{eq_H-hecke} for a definition and e.g. \cite{Cohen} for details).

In order to illustrate our results, consider the modular elliptic curve $E \simeq X_0(11)$. Assume that $-d$ is a fundamental discriminant, not divisible by $11$, and
let $ \left( \frac{-d}{\cdot} \right)$ be the quadratic Dirichlet character associated with $\Q\left(\sqrt{-d}\right)$. 
Note that $ \left( \frac{-d}{11} \right) = 1$ implies that the central critical value $L(E^{-d},1) = 0$ vanishes due to the sign of the functional equation. 
Recall that this specific setting was used as an illustration in a seminal paper by  Mazur \cite{Mazur}. In particular, \cite[Theorem 1]{Mazur} implies the following statement.

\begin{Mazur_5} 
Assume that $-d$ is fundamental and $ \left( \frac{-d}{11} \right) = -1$. 
\begin{equation} \label{Mazur_5}
\text{If} ~~ L(E^{-d},1) = 0 ~~ \text{then} ~~ H(d) \equiv 0 \pmod{5}. 
\end{equation}
\end{Mazur_5}

The importance of (non)-vanishing of the central critical value of the $L$-function is well-recognized in connection with the Birch and Swinnerton-Dyer conjecture. A conceptual proof that $5 \nmid H(d)$ (assuming $ \left( \frac{-d}{11} \right) = -1$)
implies the finiteness of the Mordell-Weil group of $E$ over $\Q\left(\sqrt{-d}\right)$ was offered in \cite{Mazur}. Alternatively, that finiteness statement can be derived from $L(E^{-d},1) \neq 0$ using the  result of Kolyvagin \cite{Kolyvagin}. Note that the results pertaining to $X_0(11)$ serve only as an illustration of the methods and results developed in \cite{Mazur}.

The statement  \eqref{Mazur_5}  was later obtained in \cite{Gue-Pan} using Rankin's method. 

Another way to obtain \eqref{Mazur_5} is to prove the modulo $5$ congruence between a certain weight $3/2$ Eisenstein series and a weight $3/2$ cusp form related to the (unique up to normalization)  weight $2$ cusp form in $S_2(\Gamma_0(11))$ via the Shimura correspondence. 
The result of Waldspurger \cite{Waldspurger} then relates the coefficients of the weight $3/2$ cusp form to the central special values $L(E^{-d},1)$. This approach was implemented by Antoniadis and Kohnen in \cite{Ant_Kohnen}. This is the approach to (non)-vanishing of $L(E^{-d},1)$ which we employ in this paper.

In \cite[Theorem 1]{Mazur} (and the alternative approaches sketched above), the prime $5$ is indispensable: it generates the Eisenstein ideal (simply speaking, it is a prime divisor of the numerator of $(11-1)/12=5/6$). All the methods indicated above do not generalize to other primes. While our congruences in Theorem \ref{thm_cong_11}  below are very similar in their appearance, 
the nature of the congruences is different. 

Abbreviate following  \cite{Zagier_traces} $B(d)=B(1,d)$. 

\begin{thm} \label{thm_cong_11}

Assume that $-d$ is fundamental and $ \left( \frac{-d}{11} \right) = -1$. 
We have that 

$L(E^{-d},1) = 0$ if and only if 
\[
H(d) = -\frac{5}{24} \lim_{n \rightarrow \infty} B(11^{2n}d),
\]
where the limit is taken $p$-adically with $p=11$.

In particular (note that $B(11^{2n}d) \equiv 2B(d) \pmod {11}$ for $n>0$), we have the following  analog of \eqref{Mazur_5}.
\[
\text{If} ~~ L(E^{-d},1) = 0 ~~\text{then} ~~ H(d) \equiv 6B(d) \pmod{11}. 
\]
\end{thm}

Note that Theorem \ref{thm_cong_11} is only an example: while the consideration of the elliptic curve $E \simeq X_0(11)$ yields some simplifications, similar congruences can be obtained using the same method for rational elliptic curves $E$ of conductor $N$ with $p||N$. One also can obtain similar congruences for the quantities $B(D,d)$ with $D>1$,
(although the Kronecker - Hurwitz class numbers show up only when $D$ is a perfect square). For example, for a fundamental discriminant $-d$ such that $ \left( \frac{-d}{11} \right) = -1$, the vanishing $L(E^{-d},1) = 0$ implies $B(5,d) \equiv 0 \pmod {11}$ (we skip the proof which is  parallel to that of Theorem \ref{thm_cong_11})

The study of congruences for the traces of singular moduli was initiated when 
Ahlgren and Ono presented in \cite[Theorem 1.1(1)]{AO} a  simple congruence which is only slightly more general than $B(1,p^{2}d) \equiv 0 \pmod p$ whenever the prime $p \geq 3$ splits in $\Q(\sqrt{-d})$. 
Numerics suggested that the congruence can be generalized to powers of $p$ and the question about possible generalizations was put by Ono in \cite[Problem 7.30]{Ono_book}.
Several generalizations \cite{Edixhoven, Jenkins, Gue730, Ahlgren} followed shortly. 
It was proved, in particular, that, for a prime $p \geq 3$,
\begin{equation} \label{equ_vanishing_old}
\left( \frac{D}{p} \right) = \left( \frac{-d}{p}  \right) \neq 0 ~~ \text{implies} ~~ B(D,p^{2n}d) \equiv 0 \pmod{p^n}
\end{equation}
Edixhoven's approach  in \cite{Edixhoven} was geometric: he used the original definition of $B(D,d)$ as traces of the special values of the modular $j$-invariant  
and the local moduli theory of ordinary elliptic
curves in positive characterstic. Other authors used nothing but the combinatorics of half-integral weight Hecke operators. 
Since we also approach the congruences from the point of view of modular forms, we define, for the purposes of this paper, the quantities $B(D,d)$ as the Fourier coefficients 
($q=\exp(2 \pi i z)$ with $\Im(z)>0$)
of weight $3/2$
modular forms
\[
g_D = q^{-D} + \sum_{d \geq 0} B(D,d) q^d.
\]
(We set $B(a,b)=0$ whenever the quantities $a$ and $b$ do not satisfy the requisite conditions, namely $B(a,b)\neq 0$ implies that $a \geq 1$ and $b \geq 0$ are  integers satisfying $a \equiv 0,1 \pmod 4$ and $b \equiv 0,3 \pmod 4$.)
It is a result of Zagier \cite{Zagier_traces} that the generating functions $g_D$ for the traces of singular moduli belong to the space of  weak modular forms of weight $3/2$ on $\Gamma_0(4)$ 
which satisfy Kohnen's plus-condition and 
constitute a basis of the space. 
In this paper, we primarily concentrate on modular forms of weight $3/2$, and we say that a formal power series 
$ f=\sum a(n)q^n$ belongs to the ``Kohnen  plus subspace" (satisfies Kohnen's plus condition) whenever 
\begin{equation} \label{eq_Kohnen_plus}
n \equiv 1,2 \pmod{4} ~~ \text{implies} ~~ a(n) = 0.
\end{equation}

\begin{rem}
In \cite{Zagier_traces}, Zagier also defined (and interpreted in terms of the traces of singular moduli) a generalization $B_m(D,d)$ as the Fourier coefficients of the holomorphic part of $g_D| T_{m^2}$, where $T_{m^2}$ is the weight $3/2$ Hecke operator. 
In this paper, in order to keep both the notations and the arguments lighter, we stick to the quantities $B(D,d)$ although analogous results may also be obtained for the quantities $B_m(D,d)$ mutatis mutandis. 
\end{rem}

Another type of congruences which hold for $B(d)$ was observed by Bruinier and Ono in \cite[Corollary 3]{BO} emploing Serre's theory of $p$-adic modular forms.
In our notations, their result may be restated as, 
for  primes $p \leq 7$ and a fundamental discriminant $-d$,
\begin{equation} \label{equ_BO_class_numbers}
\left( \frac{-d}{p} \right) =-1 ~~ \text{or $0$ implies} ~~ B(1,p^{2n}d)=B(p^{2n}d) \equiv \frac{48}{1-p}H(d) \pmod{p^n}.
\end{equation}
We want to emphasize that the condition for these congruences to hold (though only with $D=1$ and a handful of small primes)  complements the conditions in \eqref{equ_vanishing_old}.
That hints that one may benefit from packing all these congruences into a generating function: it is the power series $g_D|U^{2n}$ which should be modulo $p^n$ coefficient-wise congruent 
to a certain formal power series in $q$ 
whose coefficients are sometimes zeros, while sometimes, (at least, in the cases when $p$ is small) closely related to the class numbers. 
Here and throughout we denote by $U=U_p$ the $U$-operator acting on formal power series as 
\[
\left. \left( \sum_{n \gg -\infty} a(n) q^n \right) \right| U = \sum_{n \gg -\infty} a(pn) q^n.
\]

Our result (see Theorem \ref{thm_general} below) incorporates the congruences described above (of both types \eqref{equ_vanishing_old} and \eqref{equ_BO_class_numbers}) and yields  more congruences of similar flavor, Theorem \ref{thm_cong_11} in particular.
In order to state the result we need to introduce some notations pertaining to Kohnen's theory  \cite{Kohnen} of half-integral weight modular forms.

Let $p \geq 3$ be a prime. 
Recall that $g_D \in M^!_{3/2}(\Gamma_0(4))$ satisfies \eqref{eq_Kohnen_plus} by construction, and 
so does $g_D|U^2$.  
(Note also that the Kohnen plus subspace  in $S_{3/2}(\Gamma_0(4))$ is empty.)
Following later work by Kohnen \cite{Kohnen}, we now assign a different meaning to the  $``+"$-superscript. 
We say that a formal power series $\sum a(n)q^n$ belongs to $\gM^\pm$ whenever 
\[
n \equiv 1,2 \pmod {4} ~~ \text{or} ~~ \left(\frac{-n}{p} \right)= \mp1 ~~ \text{implies} ~~ a(n)=0.
\]
It is proved in   \cite{Kohnen} that the ``Kohnen plus subspace" $S \subseteq S_{3/2}(\Gamma_0(4p))$ of cusp forms  of weight $3/2$ on $\Gamma_0(4p)$ satisfying \eqref{eq_Kohnen_plus}  admits a Hecke-invariant splitting into a direct sum of subspaces
\[
S = \left( S \cap \gM^+ \right) \oplus  \left( S \cap \gM^- \right) 
\] 
so that for $f \in  S^\pm :=S \cap \gM^\pm  $we have
\[
f |U^2 = \mp f.
\]

Observe that the finite-dimensional $\C$-vector space $S_{3/2}(\Gamma_0(4p))$ has a rational structure: it admits a basis which consists of modular forms with their $q$-expansions in 
$\Q \llbracket q \rrbracket$. (That follows from \cite[Theorem 3.52]{Shimura} immediately because after multiplication by $\theta=\sum_{n \in \Z} q^{n^2}$ we obtain modular forms in $S_2(\Gamma_0(4p))$.) 
As subspaces, $S^\pm \subset S_{3/2}(\Gamma_0(4p))$ also admit such bases.
We thus can consider the $\Q$-vector space $S_{3/2}(\Gamma_0(4p))_\Q$ which consist of those modular forms which happen to have rational $q$-expansion coefficients. Then $S_{3/2}(\Gamma_0(4p)) = S_{3/2}(\Gamma_0(4p))_\Q \otimes \C$ and we define, for a prime $p$, 
\[
S_{3/2}(\Gamma_0(4p))_p: = S_{3/2}(\Gamma_0(4p))_\Q \otimes \Q_p ~~ \text{and} ~~ S^\pm_p := S_{3/2}(\Gamma_0(4p))_p \cap \gM^\pm.
\]

In Section \ref{sec_ZE}, we discuss a modification of the weight $3/2$ non-holomorphic Eisenstein series constructed by Zagier in \cite{Zagier_eis} in order to produce the weight $3/2$ holomorphic Eisenstein series $\tilde{\cH}$ on $\Gamma_0(4p)$, and show that, as a $q$-expansion, $\tilde{\cH} \in \gM^-$ and $\tilde{\cH}|U^2 = \tilde{\cH}$.
Specifically, 
\[
\tilde{\cH} = \frac{p-1}{12} + \sum_{\itop{d>0}{d \equiv 0,3 \pmod{4}}} \tilde{H}(d) q^d,
\]
where
\[
 \tilde{H}(d) = H(d)\left(1-  \left( \frac{-d}{p} \right)\right)- pH(d/p^2), 
\]
assuming as usual that $H(d/p^2) = 0$ for $p^2 \nmid d$.

For a formal power series in $\Q_p \llbracket q \rrbracket$
define $\ord_p \left( \sum a(n)q^n \right) = \inf_n \ord_p(a(n))$. 
We say that 
\[
\sum a(n) q^n \equiv \sum b(n) q^n \pmod {p^m} ~~ \text{whenever} ~~ \ord_p \left( \sum \left( a(n)-b(n) \right) q^n \right) \geq m.
\]
As usual, the norm $\lVert f \rVert = p^{-\ord_p (f)}$ (for $f \in \Q_p \otimes \Z_p\llparenthesis q \rrparenthesis$)
defines a topology on the space of the formal Laurent series with coefficients in $\Q_p$ (with $p$-bounded denominators) and one can consider $p$-adic limits with respect to this topology. 

Theorem \ref{thm_cong_11} will be derived from the following general statement in Section \ref{sec_conv_rate}.

\begin{thm} \label{thm_general}
Let $p \geq 3$ be a prime. Let $D=p^{2t}D_0$ with $p^2 \nmid D_0$, and
let $\epsilon = - \left( \frac{D_0}{p} \right)$. 
The $p$-adic limit 

\[
\tilde{g}_D = \lim_{n \rightarrow \infty} (-1)^{n\epsilon(\epsilon +1)/2}g_D|U^{2n} 
\]
exists.

If $\epsilon=0$ then $\tilde{g}_D=0$. 

If $D$ is not a perfect square and $\epsilon \neq 0$ then 
\[
\tilde{g}_D \in S^\epsilon_p
\]



If $D$ is a perfect square (while $p \nmid D_0$ therefore $\epsilon=1$), then 
\[
\tilde{g}_D  + \frac{24}{p-1} \tilde{\cH} \in  S^-_p
\]
\end{thm}


The case $p=2$ (which is not covered by Theorem \ref{thm_general})) is special: in this case (also only for $D=1$), a perfect result has been obtained by Boylan in \cite{Boylan} using methods which may not generalize to other primes $p$. 
In our notations, Boylan showed, in particular,  that
\begin{equation} \label{equ_Boylan_2}
g_1|U^{2n}\equiv 2\theta^3 \pmod {2^{4n+1}},
\end{equation}
where $\theta = \sum_{n \in \Z} q^{n^2}$ is the standard theta-function (the well-known simple relations between the numbers of representation of an integer as a sum of three squares and the class numbers go back to Gauss). 

While similar in its nature to Theorem \ref{thm_general}, Boylan's result  \eqref{equ_Boylan_2} is different because $\theta^3 \in M_{3/2}(\Gamma_0(4))$ does not belong to the ``Kohnen plus space", while all weight $3/2$ modular forms involved in Theorem \ref{thm_general} belong to this space. That discrepancy is due to the fact that the prime $p=2$ is peculiar in the context of half-integral weight modular forms. 

We now want to concentrate on another discrepancy between \eqref{equ_Boylan_2} and Theorem \ref{thm_general}. 
Observe that the congruence modulus in  \eqref{equ_Boylan_2} is larger than one would expect looking at \eqref{equ_vanishing_old} and \eqref{equ_BO_class_numbers} as special cases of Theorem \ref{thm_general} (it is $2^{4n+1}$, not merely $2^n$). 
Ken Ono \cite[Example 7.15]{Ono_book} noticed that, for certain small primes, the modulus of the congruences in \eqref{equ_vanishing_old} is bigger than expected, and called for an explanation. Boylan's result  \eqref{equ_Boylan_2} answers the question for $p=2$. Only a qualitative result for other small primes was obtained in \cite[Theorem 2]{Gue730}, and we now quantify it using a  method which is different from that employed in \cite{Gue730}. 

For $p=3,5,7,13$, we have that  $\dim S=0$, thus, in Theorem \ref{thm_general}, 
\[
\tilde{g}_D = \begin{cases} \frac{24}{1-p} \tilde{\cH}& \text{if $D$ is a perfect square} \\ 0 & \text{otherwise.} \end{cases}
\]
For $p=11$, we have that $\dim S^+ = 0$ (while $\dim S^- =1$), thus  in Theorem \ref{thm_general}, 
\[
\left(\frac{D}{p} \right) = -1 ~~ \text{or} ~~ 0 ~~ \text{implies} ~~ \tilde{g}_D=0. 
\]
The following result measures the (higher than expected) pace of convergence for these primes both in the case of \eqref{equ_vanishing_old} and in the case of the approximation of class numbers result \eqref{equ_BO_class_numbers}.

\begin{thm} \label{thm_3_5_7_11}
Let
\[
s= \begin{cases}
3 &  \text{if } p=3 \\
2 &  \text{if } p=5 \\
2 &  \text{if } p=7 \\
2 &  \text{if } p=11 \\
1 &  \text{if } p=13
\end{cases}
\]
We have that, for $p \in \{3,5,7,11,13\}$, 
\[
\ord_p \left(\tilde{g}_D - g_D\right)|U^{2n} = \cO(sn) ~~\text{as} ~~ n \rightarrow \infty,
\]

\end{thm}
As an example,  for $p=3$ we have that $s=3$, and numerical calculations indeed suggest that 
\[
g_1|U^{2n} \equiv \tilde{g_1} \pmod 3^{3n+3},
\]
in particular, 
\begin{equation} \label{eq_experemi}
B(3^{2n}d) = B(1,3^{2n}d) \equiv 
\left\{ \begin{array}{ll}
0 & \text{if} \left(\frac{-d}{p}\right) =1~~ \text{or} ~~ 0 \\
-24H(d) & \text{if} \left(\frac{-d}{p}\right) =-1  \\ 
\end{array} \right. \pmod{3^{3n+3}}. 
\end{equation}
Note that Theorem \ref{thm_3_5_7_11} predicts the rate of convergence ($s=3$) for all $D$, while it does not predict the specific modulus (i.e. $3n+3=sn+3$ for $D=1$).


We discuss the modified Zagier - Eisenstein series $\tilde{\mathcal{H}}$ in Section \ref{sec_ZE} and state its properties which we use later. 
In Section \ref{sec_conv_rate}, we prove Theorems   \ref{thm_cong_11}, \ref{thm_general}, and  \ref{thm_3_5_7_11}. These proofs substantially employ the theory of overconvergent half-integral weight modular forms developed by Ramsey in \cite{Ramsey2006,Ramsey,Ramsey_sh}.

\section* {Acknowledgement}
The author is grateful to Larry Rolen who attracted the author's attention to the congruences \eqref{eq_experemi} part of which was experimentally obtained in Miriam Beazer's 2021 Master thesis from Brigham Young University. The author is grateful to Kathrin Bringmann for a useful discussion related to modifications of the  Zagier - Eisenstein series. 


\section{Modified Zagier - Eisentstein series $\tilde{\cH}$} \label{sec_ZE}

Let $p \geq 3$ be a prime. Recall that we denote by $H(n)$ the Kronecker - Hurwitz class number. Namely, if $n=\gr f^2$, and $-\gr<0$ is a fundamental discriminant, then 
\begin{equation} \label{eq_H-hecke}
H(n) = \frac{h(\gr)}{w(\gr)} \sum_{m|f} \mu(m) \left( \frac{-\gr}{m} \right) \sigma_1(f/m),
\end{equation}
where $h(\gr)$ is the class number of and $w(\gr)$ is half of the number of units in $\Q(\sqrt{-\gr})$. (We set $H(n)=0$ if $n \equiv 1,2 \pmod{4}$). 
 In \cite{Zagier_eis}, the transformation law on $\Gamma_0(4)$ of the generating function  
\[
\cH = -\frac{1}{12}+\sum_{n>0} H(n) q^n,
\]
which is not a modular form, was studied. 

Besides the standard $U=U_p$ and $V=V_p$ operators acting on the formal power series in $q$ as
\[
\left. \left( \sum_{m \gg -\infty}  a(m) q^m \right) \right|U = \sum_{m \gg -\infty}  a(pm) q^m, ~~ ~~ \left. \left( \sum_{m \gg -\infty}  a(m) q^m \right) \right|V = \sum_{m \gg -\infty}  a(m) q^{pm},
\]
we will need the twist with a Dirichlet character $\xi$ defined by
\[
\left( \sum_{m \gg -\infty}  a(m) q^m \right) \otimes \xi = \sum_{m \gg -\infty}  a(m) \xi(m) q^m .
\]
The action of the weight $3/2$ Hecke operator $T_{\ell^2}$ for a prime $\ell$ can be written as 
\[
f|T_{\ell^2} = f|U_\ell^2 + f \otimes \chi_\ell + p f|V_\ell^2,
\]
where $\chi_\ell(m) = \left( \frac{-m}{\ell} \right)$.
Note that, while $\cH$ is not a modular form, it follows from \eqref{eq_H-hecke} that the formal power series $\cH$ is a common eigenseries of $T_{\ell^2}$ for all odd primes $\ell$:
\[
\cH|T_{\ell^2} = \sigma_1(\ell) \cH = (\ell+1) \cH.
\]

Constructions of closely related series $\tilde{\cH}$ which is a modular form (of weight $3/2$ on  $\Gamma_0(4p)$) are presented in \cite{Ant_Kohnen}, and, from a different perspective, in \cite[Chapter 12]{Gross}. General results on half-integral weight Eisenstein series, those in ``Kohnen plus subspace" in particular, are available in \cite{Wang_Pei}. In this Section, we present the construction of $\tilde{\cH}$ from a different perspective. 

We want to emphasize an analogy with the situation in weight $2$, where the series
\[
\mathbb{G}_2 = -\frac{1}{24} + \sum_{n \geq 1} \sigma_1(n) q^n
\]
is not a modular form, and the corresponding modification at a prime $p$ is 
\[
\tilde{\mathbb{G}}_2:= \mathbb{G}_2 - p \mathbb{G}_2|V = (1-p)\lim_{m \rightarrow \infty} \mathbb{G}_2 |U^m \in M_2(\Gamma_0(p)).
\]
(The limit is taken in the $p$-adic topology defined by the norm $||\sum a(n) q^n||_p = p^{-\inf_n \ord_p(a(n))}$.)
Observe that, while $\mathbb{G}_2$ is not a modular form, the formal power series is a common eigenseries  of all the weight $2$ Hecke operators $T_l$:
\[
\mathbb{G}_2|T_\ell=\sigma_1(\ell) \mathbb{G}_2 = (\ell+1)\mathbb{G}_2,
\]
and so is $\tilde{\mathbb{G}}_2$, with the same Hecke eigenvalues as $\mathbb{G}_2$ ,except  at $p$:
\[
\tilde{\mathbb{G}}_2|U = \tilde{\mathbb{G}}_2.
\]

We now let $\chi_p(m) = \left( \frac{-m}{p} \right)$ and define
\[
\tilde{\cH} = \cH - \cH \otimes \chi_p - p \cH|V^2.
\]
The following proposition, in particular, establishes a parallelism between our construction of $\tilde{\cH}$ and its properties and the construction and properties of $\tilde{\mathbb{G}}_2$ sketched above.
\begin{prop} \mbox{} \label{prop_tilde_H}

\begin{itemize}

 \item[\normalfont(i)]

We have that
\[
\tilde{\cH} = (1-p)\lim_{m \rightarrow \infty} \cH|U^{2m}.
\]
Furthermore, $\tilde{\cH}|T_{\ell^2} = (l+1) \tilde{\cH}$ for all odd primes $\ell \neq p$ while  $\tilde{\cH}|U^2 = \tilde{\cH}$.

\item[\normalfont(ii)]

We have that $\tilde{\cH} \in M_{3/2}(\Gamma_0(4p)) \cap \gM^-$.

\end{itemize}
\end{prop}

While the proposition and its proof are obviously related to \cite[Lemma 2.6]{BK}, the two statements are distinct: neither one implies another one.

\begin{proof}

In order to prove (i), we firstly observe that for $\ell \neq p$ the equality $\tilde{\cH}|T_{\ell^2} = (l+1) \tilde{\cH}$ follows at once rom  ${\cH}|T_{\ell^2} = (l+1) {\cH}$.

We will need  the obvious identities 
\begin{equation} \label{eq_obvious_identities}
(\cH|V^2) \otimes \chi_p = 0, ~~ (\cH \otimes \chi_p)|U^2 = 0, ~~ \cH|V^2U^2 = \cH.
\end{equation}
Furthermore, since $\cH|T_{p^2} = (p+1) \cH$, we find that 
\[ 
p\cH = - \cH + \cH|U^2+\cH \otimes \chi_p + \cH|V^2.
\]
Using the above identities we find that 
\[
\tilde{\cH}|U^2 = \cH|U^2- p\cH = \cH-\cH \otimes \chi_p - p\cH|V^2 = \tilde{\cH}.
\] 
Now let 
\[
\tilde{\cG} = \cH - \frac{1}{p} \cH \otimes \chi_p - \cH|V^2.
\]
In a similar way, we find that 
\[
\tilde{\cG}|U^2 = \cH|U^2 - \cH = p\cH - \cH\otimes \chi_p - p \cH|V^2 = p \tilde{\cG}.
\]
Note that $\tilde{\cH} - p\tilde{\cG} = (1-p) \cH$. Acting with $U^2$ on this identity repeatedly we finish our proof of  (i). 

In order to prove (ii), recall that, by \cite[Th\'eor\`eme 2]{Zagier_eis}, the non-analytic function 
\[
\cF(z) = \cH(z) + y^{-1/2} \sum_{n \in \Z} \beta(4\pi n^2 y) e^{-2\pi i n^2 z} \hspace{5mm} (z=x+iy ~~ \text{and} ~~ y>0),
\]
where $\beta(t) = \int_1^\infty u^{-3/2} e^{-ut} ~~ du$ transforms like a modular form of weight $3/2$ on $\Gamma_0(4)$.

Observe that the identity $\tilde{\cH}|U^2=\tilde{\cH}$ together with  \eqref{eq_obvious_identities} imply that
\[
\tilde{\cH} = \cH|U^2 - p\cH.
\]
It is easy to see (using $f|U^2 = p^{-2} \sum_ {m \pmod {p^2}} f(z+m/p^2)$) that the non-analytic part disappears
\[
\cF(z) |U^2 - p\cF = \cH|U^2 - p\cH = \tilde{\cH},
\]
and adopting the argument of \cite[Lemma 1]{Li}  to the half-integral weight (see \cite[Lemma 2.3(1)]{BK} for a generalization)  we obtain 
that $\tilde{\cH} \in M_{3/2}(\Gamma_0(4p))$. 

In order to finish the proof of (ii), we now prove that $\tilde{\cH} \in \gM^-$. 
Since $\tilde{\cH} = \sum \tilde{H}(N)q^n$ with
\begin{equation} \label{eq_t_H}
\tilde{H}(N) = H(N) - \left( \frac{-N}{p} \right)H(N) - pH(N/p^2),
\end{equation}
and the series $\cH$ satisfies the Kohnen plus condition \eqref{eq_Kohnen_plus} so does the series $\tilde{\cH}$. 
The extra condition at $p$ is also easy to verify:
\[
\text{if} ~~ \left( \frac{-N}{p} \right) = 1, ~~ \text{then} ~~ \tilde{H}(N) = H(N)-H(N) - p(N/p^2) =  0.
\]
\end{proof}

\begin{rem}

The fact that $\tilde{\cH}$ is the unique (up to normalization) Eisenstein series in the Kohnen plus subspace of $M_{3/2}(\Gamma_0(4p))$ (satisfying \eqref{eq_Kohnen_plus}) is well known (see  e.g. \cite[Theorem 1(I)]{Wang_Pei} for a proof).

\end{rem}



\section{Proofs of the Theorems}  \label {sec_conv_rate}

Recall that $p \geq 3$ is a prime. 
We start with proving the existence and some properties of the limit $\tilde{g}_D$ claimed in Theorem \ref{thm_general}. To begin with, we note that the recursive construction procedure of the modular forms $g_D$ presented in \cite{Zagier_traces} guarantees that $B(D,d) \in \Z$. 

The following proposition summarizes some immediate corollaries from a result of Jenkins \cite{Jenkins}. Recall from \cite[Theorem 1.1]{Jenkins} that 

\begin{multline} \label{eq_jenkins}
B(D, p^{2n}d) = p^nB(p^{2n} D,d) \\
+ \sum_{k=0}^{n-1} \left( \frac{D}{p} \right)^{n-k-1} \left( B(D/p^2, p{^{2k}}d) - p^{k+1}B(p^{2k}D  , d/p^2)  \right) \\ 
+  \sum_{k=0}^{n-1} \left( \frac{D}{p} \right)^{n-k-1}  \left( \left( \left( \frac{D}{p} \right) - \left( \frac{-d}{p} \right) \right) p^k B(p^kD,d) \right)
\end{multline}

We will derive the following proposition from \eqref{eq_jenkins}.

\begin{prop} \label{prop_lim_jenkins}

Let $D=p^{2t}D_0$ with $p^2 \nmid D_0$, and
let $\epsilon = - \left( \frac{D_0}{p} \right)$.  

\begin{itemize}

\item[\normalfont(i)]
The $p$-adic limit 
\[
\tilde{g}_D = \lim_{n \rightarrow \infty} (-1)^{n\epsilon(\epsilon +1)/2}g_D|U^{2n} 
\]
exists.

\item[\normalfont(ii)]
If $\epsilon \neq 0$, then $\tilde{g}_D  \in \gM^\epsilon$.

\item[\normalfont(iii)]
If $p|D_0$ then $\tilde{g}_D = 0$.

\item[\normalfont(iv)]
If $p^2|D$,  then we have the reduction $g_D|U^{2n} \equiv g_{D/p^2}|U^{2n-2} \pmod {p^{n-1}}$.

\end{itemize}

\end{prop}

\begin{proof}

We start with the observation that, by definition, $m = 1,2 \pmod{4}$ implies $B(D,m) =0$, and, since $p^2 \equiv 1 \pmod 4$, 
the limit series $\tilde{g}_D$, if exists, satisies \eqref{eq_Kohnen_plus}.

In order to prove (iv), assume that $p^2|D$ and observe that   \eqref{eq_jenkins} implies, since $\left(\frac{D}{p}\right)=0$, that
\[
B(D,p^{2n}d) \equiv B(D/p^2,p^{2n-2}d) - \left( \frac{-d}{p} \right) p^{n-1} B(p^{2k}D,d) \pmod{p^n}.
\]
Statement (iv) follows from that, and induction in $t$ now allows us to prove statements (i), (ii), and (iii) under the assumption that $t=0$ (i.e. $p^2 \nmid D$).

We now assume that $p||D$. Thus $\epsilon =-\left( \frac{D}{p} \right) =0$, and the sums in  \eqref{eq_jenkins}  degenerate to single terms implying that
\[
B(D,p^{2n}d) \equiv - \left( \frac{-d}{p} \right)p^{n-1} B(p^{2k}D,d) \pmod{p^n}.
\]
Thus $\lim_{n \rightarrow \infty} B(D,p^{2n}d) =0$ for every $d$.
Assertion (i) in the case $\epsilon =0$ and assertion (iii) follow from that.

We now consider the two cases $\left(\frac{D}{p} \right) =\pm 1$ separately. 

Assume firstly that $\left(\frac{D}{p} \right) =1$ (thus $\epsilon=-1$). Then \eqref{eq_jenkins} becomes
\begin{multline} \label{eq_B_1_lim}
B(D, p^{2n}d) = p^nB(p^{2n}D,d) \\ - \sum_{k=0}^{n-1} p^{k+1}B(p^{2k}D,d/p^2) + \sum_{k=0}^{n-1} \left(1- \left( \frac{-d}{p} \right) \right) p^k B(p^{2k}D,d),
\end{multline}
and implies that $\lim_{n \rightarrow \infty} B(D,p^{2n}d)$ exists for every $d$, therefore the limit $\tilde{g}_D= \lim_{n \rightarrow \infty} g_D|U^{2n}$ exists.
Furthermore, if $ \left( \frac{-d}{p} \right) = 1$ then \eqref{eq_B_1_lim} becomes $B(D, p^{2n}d) = p^nB(p^{2n}D,d)$, therefore $\lim_{n \rightarrow \infty} B(D, p^{2n}d) =0$,
in other words, $\tilde{g}_D \in \gM^-$ as required. That finishes the proof of (i) and (ii) under the assumption that $\left(\frac{D}{p} \right) =1$.

Assume now that  $\left(\frac{D}{p} \right) =-1$ (thus $\epsilon=1$).  Then \eqref{eq_jenkins} becomes
\begin{multline*} 
B(D, p^{2n}d) =  p^nB(p^{2n}D,d)   \\ - \sum_{k=0}^{n-1} (-1)^{n-k-1}p^{k+1}B(p^{2k}D,d/p^2)  -  \sum_{k=0}^{n-1}  (-1)^{n-k-1} \left( 1 + \left( \frac{-d}{p} \right) \right) p^k B(p^{2k}D,d) \\
\equiv (-1)^{n} \sum_{k=0}^{n-1} (-1)^k p^k \left( pB(p^{2k}D,d/p^2) + \left( 1+ \left(\frac{-d}{p} \right) \right)B(p^{2k}D,d) \right) \pmod{p^n},
\end{multline*}
and implies that $\lim_{n \rightarrow \infty} (-1)^nB(D,p^{2n}d)$ exists for every $d$, therefore $\tilde{g}_D= \lim_{n \rightarrow \infty} (-1)^ng_D|U^{2n}$ exists.
Furthermore,  if $ \left( \frac{-d}{p} \right) = -1$, then again $B(D, p^{2n}d) = p^nB(p^{2n}D,d)$, therefore $\lim_{n \rightarrow \infty} B(D, p^{2n}d) =0$,
in other words, $\tilde{g}_D \in \gM^+$ as required. That finishes the proof of (i) and (ii) under the assumption that  $\left(\frac{D}{p} \right) =-1$.
\end{proof}

\begin{proof}[Proof of Theorem \ref{thm_general}]
Reduction to the case $t=0$ follows from Proposition \ref{prop_lim_jenkins}(iv). 
The case $\epsilon=0$ follows at once from Proposition \ref{prop_lim_jenkins}(iii).
We thus now assume that $t=0$ and $\epsilon \neq 0$ (i.e. $p \nmid D_0=D$). 

Define for convenience 
\[
\delta_\square(D) = \begin{cases} 1 & \text{if $D$ is a perfect square} \\ 0 & \text{otherwise} \end{cases}
\]
It follows from \cite[Theorem 4]{Zagier_traces} that the constant term of the $q$-expansion  of $g_D$ vanishes whenever $D$ is not a perfect square, and equals $-2$ otherwise.
Thus
\[
h_D = g_D + \delta_\square(D)\frac{24}{p-1} \tilde{\cH}
\]
is a weight $3/2$ weak modular form on $\Gamma_0(4p)$ (satisfying \eqref{eq_Kohnen_plus}).

In order to finish the proof of Theorem \ref{thm_general}, it now suffices to show that 
\begin{equation} \label{eq_limit_2}
\tilde{h}_D=\lim_{n \rightarrow \infty} (-1)^{n \epsilon (\epsilon+1)/2} h_D|U^{2n} \in S_{3/2}(\Gamma_0(4p))_p
\end{equation}
since the existence of the limit  $\tilde{h}_D \in \gM^\epsilon$ follows from Propositions \ref{prop_lim_jenkins} and \ref{prop_tilde_H}.

Since the constant term of $h_D$ at infinity is zero, and $\tilde{\cH}$ is the unique (by \cite[Theorem 1(I)]{Wang_Pei}) Eisenstein series in the Kohnen plus subspace of $M_{3/2}(\Gamma_0(4p))$ (satisfying \eqref{eq_Kohnen_plus}),
the constant term of $h_D$ at every cusp is zero. 

We now consider $h_D \in \tilde{S}_{3/2}   (4p, \Q_p, p^{-r})$ as an element of the 
cuspidal subspace of the Banach space of overconvergent weight $3/2$ modular forms  defined in \cite{Ramsey}. 
As usual, we fix the overconvergence rate $1/(p+1) < r < p/(p+1)$ arbitrarily.
The operator $U_{p^2}$ acting on $\tilde{S}_{3/2}   (4p, \Q_p, p^{-r})$ is constructed geometrically in \cite[Section 5]{Ramsey} and is shown in \cite[Proposition 5.5]{Ramsey} to coincide with our $U^2$ on $q$-expansions. Repeated application of $U_{p^2}$ projects to the finite-dimensional subspace of the Banach space which consists of the cusp forms of zero slope. 
Using \cite[Lemma 7.2.1]{Hida_book} we conclude that, 
 in the topology of the Banach space (see \cite[Section 2.2]{Ramsey} for the definition of the Banach norm), the limit as $\tilde{F}_D:= \lim_{n \rightarrow \infty} h_D|U_{p^2}^{n!}$  exists. 
 (Note that  the operator $\lim_{n \rightarrow \infty} U_{p^2}^{n!}$ is  the standard ordinary projector.) 
 By  \cite[Theorem 6.1]{Ramsey} $\tilde{F}_D$ must be a classical cusp form, which, by \cite[Proposition 6.2]{Ramsey}, belongs to the space of overconvergent $p$-adic modular forms constructed in \cite{Ramsey2006}, thus the $q$-expansion of $\tilde{F}_D$ at infinity coincides with that of an element of $S_{3/2}(\Gamma_0(4p))_p$. 
Since the convergence in the Banach space topology implies the convergence in the topology determined by the norm $\lVert f \rVert = p^{-\ord_p (f)}$ (for $f \in \Q_p \otimes \Z_p\llparenthesis q \rrparenthesis$), we conclude that  $\tilde{F}_D = \tilde{h}_D$. (Note that we can get away with the limit \eqref{eq_limit_2} defining $\tilde{h}_D$ 
instead of the more general ordinary projector $\lim_{n \rightarrow \infty} h_D| U_{p^2}^{n!}$ because of the specific of our setting: all $p$-ordinary eigenforms in $S$  are new at  $p$, thus have the $U^2$ eigenvalues of $\pm1$.) 
Thus $\tilde{h}_D \in S_{3/2}(\Gamma_0(4p))_p$ as required.
\end{proof}




We now derive Theorem \ref{thm_cong_11} as a special case.

\begin{proof}[Proof of Theorem\ref{thm_cong_11}]
Let $p=11$, and assume that $-d$ is fundamental. 

Since $E \simeq X_0(11)$ is modular, the vanishing of $L(E^d,1)$ is equivalent to the vanishing of the central special value of the twisted $L$-function 
$L\left(f, \left( \frac{-d}{\cdot} \right) \right)$ associated with the (unique up to a normalization) modular form $F \in S_2(\Gamma_0(11))$. 
We observe following \cite{Ant_Kohnen} that the space $S^-$ is also one-dimensional generated by $G \in S^-$, and the modular forms $F$ and $G$ are related by the Shimura correspondence (see also \cite{Kohnen} for details). 
Explicitly (see \cite{Ant_Kohnen}), in terms of the $q = \exp(2 \pi i z)$ variable, 
\[
G= \sum_{n>0}c(n) q^n = \left( \eta(q^2) \eta(q^{22}) \theta(q^{11}) \right) |U_4 ~~ \text{and} ~~ F= \eta(q)^2 \eta(q^{11})^2,
\]
where $\eta(q) = q^{1/24} \prod_{n \geq 1} (1-q^n)$ is the Dedekind $\eta$-function, and $\theta(q) = \sum_{n \in \Z} q^{n^2}$ is the theta-function.
It is nowadays standard to apply the result of Waldspurger \cite{Waldspurger} which, in this situation allows one to conclude that, for a fundamental discriminant $-d<0$,
the central critical value  $L(E^{-d},1)=0$ vanishes if an only if $c(d)=0$.
On the other hand, Theorem  \ref{thm_general} (with $D=1$, therefore $\epsilon=-1$) implies that 
\begin{equation} \label{eq_la}
\lim_{n \rightarrow \infty} g_1|U^{2n} = \tilde{g}_1 = -\frac{24}{11-1} \tilde{\cH} + \lambda G
\end{equation}
with some $\lambda \in \Q_p$.
We claim that $\lambda \neq 0$. In order to support this claim, we calculate the coefficient of $q^3$ of the left and right sides of \eqref{eq_la}.
It follows from \eqref{eq_B_1_lim} that the coefficient of $q^3$ in $\tilde{g}_1$ is 
\[
\lim_{n \rightarrow \infty} B(1,3p^{2n}) = 2 \sum_{k=0}^\infty p^k B(p^{2k},3) \equiv B(1,3) \pmod p
\] 
(with $p=11$).

The coefficient of $q^3$ in $\tilde{\cH}$, taking into the account \eqref{eq_t_H} and $\left( \frac{-3}{11} \right) = -1$, is
\[
\tilde{H}(3) = H(3) - H(3) \left( \frac{-3}{11} \right) = 2H(3) = 2/3
\]
Since 
\[
248 = B(1,3) \not\equiv  -\frac{24}{10} \cdot \frac{2}{3} = -\frac{8}{5} \pmod {11},
\]
we must have that $\lambda \neq 0$ in \eqref{eq_la}.

Since $\lambda \neq 0$, for a fundamental discriminant $-d$, it follows from \eqref{eq_la} that $c(d)=0$ if and only if 
\[
\lim_{n \rightarrow \infty} B(1,11^{2n}d) = -\frac{12}{5} \tilde{H}(d) = -\frac{24}{5} H(d),
\]
the last equality follows from \eqref{eq_t_H} and the condition $\left( \frac{-d}{11} \right) = -1$.

Finally, the congruence $B(1, 11^{2n}d) = B(11^{2n}d) \equiv 2B(d) = 2 B(1,d) \pmod {11}$ for $n>0$ and a fundamental discriminant $-d$ with  $\left( \frac{-d}{11} \right) = -1$ follows at once from \eqref{eq_B_1_lim}.

\end{proof}

\begin{proof}[Proof of Theorem \ref{thm_3_5_7_11}]
In order to prove Theorem \ref{thm_3_5_7_11}, we observe that, by  \cite[Proposition 5.1]{Ramsey}, the operator 
$U_{p^2}$ acting on $\tilde{S}_{3/2}   (4p, \Q_p, p^{-r})$  for $1/(p+1) < r < p/(p+1)$ is compact, therefore has a discrete spectrum, and invoke the asymptotic expansion lemma \cite[Proposition 1]{Gouvea_Mazur}. Specifically, in our case, since the closedness in the Banach space norm implies the closedness of the $q$-expansions, \cite[Proposition 1]{Gouvea_Mazur} implies that
\[
\ord_p \left(\tilde{g}_D - g_D\right)|U^{2n} = \cO(sn) ~~\text{as} ~~ n \rightarrow \infty,
\]
where $s$ is the smallest non-zero slope in the spectrum of  $U_{p^2}$ acting on the Banach space $\tilde{S}_{3/2}   (4p, \Q_p, p^{-r})$. It follows from \cite[Lemma 5.3]{Ramsey_sh} that every slope in the spectrum of  $U_{p^2}=U^2$ acting on 
$\tilde{S}_{3/2}   (4p, \Q_p, p^{-r})$ appears as a slope of the spectrum of the $U_p = U$ acting on the Banach space 
$\tilde{S}_{2}   (p, \Q_p, p^{-r})$ of weight $2$ overconvergent forms. Thus the problem to find the minimal slope $s$ reduces to calculating the minimal slope in the space 
$S_{2+(p-1)p^A}(SL_2(\Z))$ of classical cusp forms of trivial level and weight $2+(p-1)p^A$ with $A$ big enough. We use  the explicit estimate obtained by Wan \cite{Wan} in order to pick the appropriate values of $A$, and computer calculations to find the minimal slope $s$ in the specific spaces of the integral weight classical cusp forms. For the latter, we used the GP calculator  \cite{PARI2}.
That finishes the proof of Theorem \ref{thm_3_5_7_11}. \end{proof}

\end{document}